\documentclass[3p,times]{elsarticle}

\usepackage{ecrc}

\volume{00}

\firstpage{1}

\journalname{[Insert Journal Name Here]}

\runauth{Sch\"utz and Seal}

\jid{[Inerst Journal ID here]}

\usepackage{amssymb}
\usepackage{bibref}

\usepackage[figuresright]{rotating}

\usepackage{amsmath,amsthm,amssymb,graphicx,stmaryrd}
\usepackage{tikz}
\usepackage{hyperref}
\usepackage{xcolor}
\hypersetup{
	colorlinks,
	linkcolor={red!50!black},
	citecolor={blue!50!black},
	urlcolor={blue!80!black}
}
\usepackage{pgfplots}
\usepackage{placeins}
\usepackage{color}
\usepackage{pdfpages}
\usetikzlibrary{shapes,arrows,calc,external,matrix,positioning,patterns}

\DeclareMathOperator{\R}{\mathbb{R}}

\DeclareMathOperator{\Z}{\mathbb{Z}}

\newcommand{\Dt}{\mathrm{d}t}

\renewcommand{\O}{\mathcal{O}}

\newcommand{\eps}{\varepsilon}

\newcommand{\dt}{\Delta t}

\newcommand{\Phie}{\Phi_{\text{E}}}
\newcommand{\Phii}{\Phi_{\text{I}}}

\newcommand{\Tend}{T_{end}}

\newcommand{\I}{\mathcal{I}}

\newtheorem{postulate}{Postulate}
\newtheorem{definition}{Definition}
\newtheorem{algorithm}{Algorithm}
\newtheorem{theorem}{Theorem}
\newtheorem{lemma}{Lemma}

\newtheorem{remark}{Remark}

\newcommand{\lims}[2]{#1_{(#2)}}
\newcommand{\alg}[2]{#1^{[#2]}}

\renewcommand{\dot}[1]{\overset{\boldsymbol .}{#1}}

\pgfplotsset{compat=1.14}   

\begin{document}

\begin{frontmatter}



\dochead{}

\title{An asymptotic preserving semi-implicit multiderivative solver}

\author{Jochen Sch\"utz$^\dagger$}
\author{David C. Seal$^*$ }

\address{$\dagger$Faculty of Sciences, Hasselt University, Agoralaan Gebouw D, BE-3590 Diepenbeek \\
*United States Naval Academy, Department of Mathematics, 572C Holloway Road, Annapolis, MD 21402, USA
}

\begin{abstract}
  In this work we construct a multiderivative implicit-explicit (IMEX) scheme for a class of stiff ordinary differential equations.  Our solver is high-order accurate and has an asymptotic preserving (AP) property.  The proposed method is based upon a two-derivative backward Taylor series base solver, which we show has an AP property.  Higher order accuracies are found by iterating the result over a high-order multiderivative interpolant of the right hand side function, which we again prove has an AP property.  Theoretical results showcasing the asymptotic consistency as well as the high-order accuracy of the solver are presented.  In addition, an extension of the solver to an arbitrarily split right hand side function is also offered.  Numerical results for a collection of standard test cases from the literature are presented that support the theoretical findings of the paper.
\end{abstract}

\begin{keyword}
Multiderivative \sep IMEX \sep singularly perturbed ODE \sep asymptotic preserving



\end{keyword}

\end{frontmatter}


\section{Introduction}
\label{sec:introduction}

In this work we consider the numerical approximation of the system of differential equations 
\begin{align}
  \label{eq:vdp}
  {y}'(t) &= {z}(t), \qquad {z}'(t) = \frac{g(y(t),z(t))}{\eps}, \quad 0 \leq t \leq T,
\end{align}
where $g: \R^2 \rightarrow \R$ is a smooth function.  Our goal is to construct and analyze a high-order, semi-implicit, mulitderivative asymptotic preserving solver for this class of equations.  This problem is
equipped with initial conditions at time $t = 0$ defined by
\begin{align}
 \label{eq:initial}
 y(0) = y^0, \qquad z(0) = z^0,
\end{align}
and we assume $0<\eps \ll 1$.  The presence of this stiff relaxation parameter turns the problem into a singularly perturbed equation. At this point there is a vast amount of literature on problems of this kind.  We refer the interested reader to the classical books \cite{HaiWan,KC,robert2012singular} and the references therein for an overview on both analysis and applications of this class of problems.

Provided that the initial conditions are appropriately chosen, Eqn. \eqref{eq:vdp} exhibits a multiscale behavior due to the presence of the stiff relaxation parameter.  Modern solvers leverage this behavior by splitting the equation into `stiff' and `non-stiff' terms for efficient implementations of a numerical discretization. This procedure leads to the now famous implicit-explicit, or IMEX class of methods \cite{Roux1979,Crouzeix1980,Ascher1995,Ascher1997,Bo07,Bos09,RuBosc09}.  IMEX methods can usually be classified as either being an IMEX Runge-Kutta method, or an IMEX multistep method.  Recent work has included definitions for IMEX General Linear Methods (GLMs) \cite{Zhang2014}.

A distinct class apart from the aforementioned time integrators include multiderivative methods \cite{HaWa73}.  These solvers have recently been proven to be promising alternatives to classical Runge-Kutta and multistep schemes, and they are currently experiencing a renaissance with regard to their application to PDEs \cite{TC10,Seal13,SSJ2017}. Much like Taylor approximations, these solvers work with not only with the first derivative, $y' \equiv z$ and $z' \equiv \frac g \eps$, respectively, but they also leverage higher time derivatives of the unknowns. In doing so, the formulation of a solver becomes more intricate, but the tradeoff is that it makes the formulation more local, meaning additional information about the ODE can be garnered from each time point or stage value in the solver.  This is particularly beneficial for modern high-performance computing architectures, as these solvers have the potential to reduce the memory overhead. To the best of our knowledge, there currently exists no extensions of multiderivative methods to IMEX schemes, which is the subject and aim of this work.


We propose a fourth-order multiderivative IMEX scheme that is based in part on an idea similar to the spectrally deferred correction (SDC) method \cite{DuttGreenRokh00,Minion2003,boscarino2017implicit}. Our solver makes use of a second-order two-derivative IMEX Taylor series base solver, which we carefully construct in such a way that it contains an asymptotic preserving property.  This solver serves as a `predictor', upon which iterations are performed on a higher-order scheme that lets us pick up the order of accuracy of the method.  In the current work we stop at fourth-order accuracy but our method can be extended to higher orders based upon what is presented here.  In this work we analyze the method with respect to its asymptotic properties, we show that it is asymptotically consistent \cite{Jin99,Jin2012} with the continuous asymptotics, and we show numerical results indicating we do indeed find high-order accuracy. Finally, we show that order reduction common to many IMEX solvers for stiff relaxation problems can be overcome in some situations with our proposed method.

The paper is structured as follows: In Sec.\,\ref{sec:equation}, we discuss to the necessary theoretical details for equation \eqref{eq:vdp}, including restrictions on $g$ and on the initial data. In Sec.\,\ref{sec:imexmd}, the proposed method is presented and consistency in the discretization parameter $\dt$ is shown. This is followed by an \emph{asymptotic} consistency analysis in Sec.\,\ref{sec:ac}, i.e., consistency in the singular parameter $\eps$ is demonstrated. Numerically, in Sec.\,\ref{sec:aa}, we discuss the phenomenon of order reduction and show how the algorithm can be used to address common challenges found with stiff IMEX solvers.  As not all singularly perturbed ODEs are of the form defined in \eqref{eq:vdp}, we extend the method to more general ordinary differential equation in Sec. \ref{sec:stability}, where we include stability results for a prototype equation that contains an additive right hand side. Finally, we offers conclusions and an outlook for future work in Sec.\,\ref{sec:conout}.

\section{The underlying equations: Necessary analytic properties}
\label{sec:equation}

We consider the equation as given in \eqref{eq:vdp}. 
Although the solution to the equations $y$ and $z$ depend on $\eps$, we only bring this up when necessary in our analysis. In addition, the dependence on the time variable is typically only written explicitly if needed.
Furthermore, we
need some (frequently used) assumptions on $\partial_z g$, which are related to stable manifolds of the $\eps \rightarrow 0$ limit of the solution \cite{HaiWan}. 
\begin{postulate}\label{pos:dzg}
Assume that $g$ is a smooth function, and that $\partial_z g(y,z) < -c$ for all $y, z \in \R$, with $c > 0$ being a positive constant. 
\end{postulate}
\begin{remark}
 Note that many equations only fulfill the assumption locally. This would be enough for our purposes, a local truncation of $g$ would then be sufficient. 
\end{remark}
As our solver is a multiderivative solver, we need information about the second derivative of the unknown variables.  A straightforward computation gives use the following Lemma.

\begin{lemma}\label{la:sndyz}
  For the solutions $y$ and $z$ to \eqref{eq:vdp}, there holds
  \begin{equation} 
   {y''} = \frac{g(y,z)}{\eps}, \qquad \text{and} \qquad {z''} = \frac{1}{\eps} \underbrace{\nabla g(y,z) \cdot \left(\begin{matrix} z \\ \frac{g(y,z)}{\eps} \end{matrix} \right)}_{=:\dot g(y,z)}.
  \end{equation}
  By $\nabla g$, we denote the vector $\nabla g := \left(\partial_y g,  \partial_z g \right)$, i.e., differentiation with respect to $y$ and $z$. 
\end{lemma}
For the sake of having more explicit error constants, we assume that $g$ and $\dot g$ are Lipschitz. 
\begin{postulate}\label{pos:lipschitz}
 We assume  that $g$ and $\dot g$ are Lipschitz, and that there holds
 \begin{equation}
  \|g(y_1,z_1) - g(y_2,z_2) \| \leq L_g \left\|\begin{matrix} y_1-y_2 \\ z_1-z_2\end{matrix}\right\|, \qquad
  \|\dot g(y_1,z_1) - \dot g(y_2,z_2) \| \leq \frac{L_{\dot g}}{\eps} \left\|\begin{matrix} y_1-y_2 \\ z_1-z_2\end{matrix}\right\|,
 \end{equation}
 where $L_g$ and $L_{\dot g}$ are constants independent of $\eps$. 
 Note that the $\eps^{-1}$ scaling for $\dot g$ is the correct one to use here (cf.\ La.~\ref{la:sndyz}).
\end{postulate}

Problems of the form \eqref{eq:vdp} have an interesting structure in the asymptotic limit as $\eps \to 0$.  The starting point for analyzing these problems is to first assume that both $y$ and $z$ can be written in terms of a Hilbert expansion in $\eps$, i.e., 
\begin{align} 
\label{eq:hilbert}
\begin{split}
 y(t) = \lims y 0(t) + \eps \lims y 1 (t) + \eps^2 \lims y 2(t) + \O(\eps^3), \\
 z(t) = \lims z 0(t) + \eps \lims z 1 (t) + \eps^2 \lims z 2(t) + \O(\eps^3). 
\end{split}
\end{align}
(For a proof of this identity and more related results, see \cite{HaiWan} and the references contained within.) By substituting these expansions into the system of differential equations, equations that each of the $y_{(0)}, z_{(0)}, y_{(1)}, z_{(1)}, \dots$ must satisfy can be derived.  For example, substituting into the first equation in \eqref{eq:vdp}, we find that
\begin{equation}
\lims y 0'(t) + \eps \lims y 1'(t) + \eps^2 \lims y 2'(t) + \O(\eps^3) = \lims z 0(t) + \eps \lims z 1 (t) + \eps^2 \lims z 2(t) + \O(\eps^3).
\end{equation}
As $\eps \to 0$, we have the identity $\lims y 0' = \lims z 0$.  After substituting the Hilbert expansion \eqref{eq:hilbert} into the second equation, we find
\begin{equation}
\lims {z'} 0(t) + \eps \lims {z'} 1 (t) + \eps^2 \lims {z'} 2(t) + \O(\eps^3) = \frac{1}{\eps}  g\bigg(\lims y 0 + \eps \lims y 1 + \O(\eps^2), \lims z 0 + \eps \lims z 1 + \O(\eps^2)\bigg).
\end{equation}
A Taylor expansion of $g$ is given by
\begin{align}
\label{eq:taylor_on_g}
g\bigg(\lims y 0 + \eps \lims y 1 + \O(\eps^2), \lims z 0 + \eps \lims z 1 + \O(\eps^2)\bigg) 
= g(\lims y 0, \lims z 0) + \eps \nabla g(\lims y 0, \lims z 0) \cdot \left(\begin{matrix} \lims y 1 \\ \lims z 1 \end{matrix}\right) + \O(\eps^2), 
\end{align}
which yields with \eqref{eq:vdp}
\begin{equation}
\eps \lims {z'} 0 + \O(\eps^2) = g( \lims y 0, \lims z 0) + \eps \nabla g(\lims y 0, \lims z 0) \cdot \left(\begin{matrix} \lims y 1 \\ \lims z 1 \end{matrix}\right) + \O( \eps^2 ),
\end{equation}
from which one can conclude 
\begin{equation} 
 0 = g( \lims y 0, \lims z 0 ).
\end{equation}
That is, under the assumption that $y$ and $z$ had Hilbert expansions, we have that the first order terms $\lims y 0$ and $\lims z 0$ fulfill the differential-algebraic equation 
\begin{align}
 \label{eq:vdp_lim}
 \lims y 0' = \lims z 0, \qquad 0 = g(\lims y 0, \lims z 0).
\end{align}
Eqn. \eqref{eq:vdp_lim} is the limit equation to first order, it is of course possible to extend this procedure. 
For example, based on the Taylor expansion for $g$ given in \eqref{eq:taylor_on_g}, we find that to second order there holds 
\begin{equation}
 \label{eq:vdp_lim2}
 \lims y 1' = \lims z 1, \qquad \lims z 0' = \nabla g(\lims y 0, \lims z 0) \cdot \left(\begin{matrix} \lims y 1 \\ \lims z 1 \end{matrix}\right). 
\end{equation}
These equations can be combined to produce identities for the higher derivatives not only of the original functions $y$ and $z$, but also of the asymptotic quantities $\lims y 0$ and the like.  For example, we have the following result.
\begin{lemma}
	The solutions $\lims y 0$ and $\lims z 0$ to \eqref{eq:vdp_lim} satisfy
	\begin{equation}
	\lims y 0 '' = \frac{-\partial_y g(\lims y 0, \lims z 0) \lims z 0}{\partial_z g(\lims y 0, \lims z 0)}.
	\end{equation}
\end{lemma}
\begin{proof}
	Differentiate \eqref{eq:vdp_lim} with respect to time and make use of the chain rule.
\end{proof}
Finally, we further note that thanks to $g(\lims y 0, \lims z 0) = 0$ in Eqn.\,\eqref{eq:vdp_lim}, we also have
\begin{align} 
 0 &= \frac{d}{dt} g(\lims y 0, \lims z 0) = \nabla g(\lims y 0, \lims z 0) \cdot \left(\begin{matrix} \lims y 0' \\ \lims z 0' \end{matrix}\right) 
 \stackrel{\eqref{eq:vdp_lim},\eqref{eq:vdp_lim2}}= \nabla g(\lims y 0, \lims z 0) \cdot \left(\begin{matrix} \lims z 0 \\ \nabla g(\lims y 0, \lims z 0) \cdot \left(\begin{matrix} \lims y 1 \\ \lims z 1 \end{matrix}\right) \end{matrix}\right) 
\end{align}
This property is important in our asymptotic analysis.

The proposed numerical scheme makes use of not only the right-hand side of \eqref{eq:vdp}, but also on the temporal derivative thereof. Intuitively, it is therefore reasonable to extend the concept of \emph{well-preparedness} \cite{SKN15} to cope also with the limit equation to second order. 

\begin{definition}[Well-preparedness]\label{def:wellprep}
 We call the initial conditions $(y_0,z_0)$ \emph{well-prepared} if they possess a Hilbert expansion.  That is, there exist a collection of unique functions
 $\lims {y^0} 0, \lims {y^0} 1, \dots$ and 
 $\lims {z^0} 0, \lims {z^0} 1, \dots$ for which
 the initial conditions can be expanded as
 \begin{equation}
  y^0 = \lims {y^0} 0 + \eps \lims {y^0} 1 + \O(\eps^2), \quad 
  \text{and} \quad 
  z^0 = \lims {z^0} 0 + \eps \lims {z^0} 1 + \O(\eps^2).
 \end{equation}
Furthermore, we must have
 \begin{align}
  \label{eq:wellprep}
  g(\lims {y^0} 0,\lims {z^0} 0) = 0, \qquad 
  \nabla g(\lims {y^0} 0, \lims {z^0} 0) \cdot \left(\begin{matrix} \lims {z^0} 0 \\ \nabla g(\lims {y^0} 0, \lims {z^0} 0) \cdot \left(\begin{matrix} \lims {y^0} 1 \\ \lims {z^0} 1 \end{matrix}\right) \end{matrix}\right) = 0.
 \end{align}
\end{definition}
\begin{remark}
	A couple of comments regarding the definition of well-prepared initial conditions are in order.
 \begin{itemize}
  \item The well-preparedness property is a necessary condition that the solution to the ODE defined in \eqref{eq:vdp} has a Hilbert expansion given by \eqref{eq:hilbert}. 
  \item Typically, only the first equation in Eqn.\,\eqref{eq:wellprep} is enforced. However, the standard test cases shown in literature, see, e.g., Section 5 in \cite{Bos09}, 
  fulfill this property. (In fact, using \eqref{eq:vdp_lim2} and higher-version thereof will automatically yield the initial conditions used in \cite{Bos09}.)
  \item In \cite{boscarino2017implicit}, Boscarino and collaborators use a more general version of these initial conditions; the initial conditions we are using are called 'well-prepared to order one' in their nomenclature. 
 \end{itemize}

\end{remark}

\section{The multiderivative implicit-explicit (MD-IMEX) method}
\label{sec:imexmd}

We now describe the numerical method proposed in this work, first starting with a definition of the scheme for the class of equations presented in \eqref{eq:vdp}.  Extensions of this method to larger classes of ODEs are discussed in Sec.~\ref{sec:stability}, but much of the notation that we define here remains the same.

To begin, we start with a mesh spacing
\begin{equation}
0 = t^0 < t^1 < t^2 < \dots < t^N = \Tend, \quad t^{n+1} - t^n = \dt, \quad n=0, 1, \dots, N-1,
\end{equation}
of the time domain $[0,\Tend]$.  We seek discrete numerical approximations $y^n \approx y(t^n)$ and $z^n \approx z(t^n)$ to the exact solutions $y(t)$ and $z(t)$ of \eqref{eq:vdp} at each time point $t=t^n$.  For the sake of exposition, we restrict our attention to uniform time steps, but this work can certainly be extended to a non-uniform (or adaptive) time grid.

One well known method for updating the solution to this problem would be to apply the Trapezoidal rule to approximate the integral of the right hand side to produce a second-order solver via:
\begin{align}
y^{n+1} &:= y^n + \frac{\dt}{2} \left( z^n + z^{n+1} \right) 
	\approx y^n + \int_{t^n}^{t^{n+1}} y'(t)\,\Dt = y^n + \int_{t^n}^{t^{n+1}} z(t)\,\Dt, \\
z^{n+1} &:= z^n + \frac{\dt}{2 \eps} \left( g^n + g^{n+1} \right) 
	\approx z^n + \int_{t^n}^{t^{n+1}} z'(t)\,\Dt = z^n + \int_{t^n}^{t^{n+1}} \frac{g(y(t),z(t))}{\eps}\,\Dt,
\end{align}
where $g^n := g( y^n, z^n )$, for $n=0,1,\dots, N$.  A lesser well-known strategy is to use a fourth-order integral approximation to the right hand side that makes use of not only the first, but also the second derivative of the right hand side:
 	\begin{align}
y^{n+1} &:= y^n + \frac{\dt}{2} \left(z^n + z^{n+1} \right) + \frac{\dt^2}{12 \eps}\left(g^n-g^{n+1} \right)
	\approx y^n + \int_{t^n}^{t^{n+1}} z(t)\,\Dt, \\
z^{n+1} &:= z^n + \frac{\dt}{2\eps} \left( g^n + g^{n+1} \right) + \frac{\dt^2}{12\eps} \left( {\dot g^n -\dot g^{n+1}} \right)
	\approx z^n + \int_{t^n}^{t^{n+1}} \frac{g(y(t),z(t))}{\eps}\,\Dt,
\end{align}
where the total time derivative of the right hand side of $z$ is defined as
\begin{equation}
{\dot g}^{n} := \nabla g^n \cdot \left( \begin{matrix} z^n \\ \frac 1 \eps  g^n \end{matrix} \right), \quad n=0,1,2,\dots,N.
\end{equation}
Of course neither of these methods are semi-implicit.  Not only that, but we also seek a high-order method.  Therefore, we start with the latter of these two, but we make a modification to the solver so that it becomes a semi-implicit, rather than a fully implicit solver.

Our proposed method is as follows.
\begin{algorithm}\label{alg:mdvdp}
 For the solution of \eqref{eq:vdp}, we propose the following semi-implicit iterative IMEX method to advance the solution from time $t = t^n$ to time $t = t^{n+1}$:
 \begin{enumerate}
 	\item \textbf{Predict.}  Given the solution $(y^n,z^n)$, we compute a second-order IMEX Taylor approximation
 	\begin{align} 
 	\alg y 0 &:= y^n + \dt z^n + \frac{\dt^2}{2\eps} \underbrace{g(y^n,z^n)}_{=:g^n}, \\
 	\alg z 0 &:= z^n + \frac{\dt}{\eps} \underbrace{g(\alg y 0,\alg z 0)}_{=:\alg g 0} - \frac{\dt^2}{2\eps} \underbrace{\nabla g(\alg y 0, \alg z 0) \cdot \left(\begin{matrix} \alg z 0 \\ \frac 1 \eps g(\alg y 0, \alg z 0) \end{matrix}\right)}_{:= \alg {\dot g} 0}
 	\end{align}
 	for the unknowns $\alg y 0$ and $\alg z 0$ that are our initial guesses for an approximation to $y(t^{n+1})$ and $z(t^{n+1})$.  Note that this discretization is based upon a second-order \emph{forward} Taylor series in $y$ and a second-order \emph{backward} Taylor series in $z$.  (In due course, we show that the presence of the implicit second order terms is important in the asymptotic analysis of the method.)
 	\item \textbf{Correct.}  Based on this initial step, for $0 \leq k \leq k_{\max}-1$ we solve
 	\begin{align}
 	\alg y {k+1} &:= y^n + \frac{\dt}{2} \left(z^n + \alg z k\right) + \frac{\dt^2}{12 \eps}\left(g^n-\alg g k \right), \\
 	\alg z {k+1} &:= z^n + \frac{\dt}{\eps} \left( \alg g {k+1} - \alg g k\right) - 
 	\frac{\dt^2}{2\eps} \left(\alg {\dot g} {k+1} - \alg {\dot g} {k} \right)
 	+ \frac{\dt}{2\eps} \left( g^n + \alg g k\right) + \frac{\dt^2}{12\eps} \left(\alg {\dot g^n -\dot g} k \right),
 	\end{align}
 	for $\alg y {k+1}$ and $\alg z {k+1}$.  Note that for ease of notation, we define
 	\begin{align}
 	\alg g k := g( \alg y k, \alg z k), \quad \text{and} \quad 
 	\alg {\dot g} {k} := \nabla \alg g k \cdot \left( \begin{matrix} \alg z k \\ \frac 1 \eps \alg g k \end{matrix} \right), \quad k=0,1,\dots k_{\max}.
 	\end{align}
%
 	\item \textbf{Update.}  The update for the solution is defined as
 	\begin{align*}
 	y^{n+1} := \alg y {k_{\max}}, \qquad z^{n+1} := \alg z {k_{\max}}.
 	\end{align*}
 \end{enumerate}

\end{algorithm}

Consistency and stability are of central importance for any numerical discretization of a differential equation.  Furthermore, for aymptotic-preserving (AP) schemes, the asymptotic stability and accuracy (as $\eps \to 0$) are of paramount import, as these are the defining features of any AP numerical solver.  We analyze the latter two central properties in the forthcoming sections, but first we address the consistency of the numerical method by looking at the order of accuracy of the solver (as a fixed function of $\eps > 0$) and letting $\dt \to 0$.  Stability is investigated in the numerical results section where we consider a prototypical linear case after defining the appropriate extension of this solver to problems with an additive right hand side.

\begin{remark} \label{rmk:consistency}
 In every iteration step, $\alg y k$ and $\alg z k$ are approximations to $y(t^{n+1})$ and $z(t^{n+1})$, respectively, of order $\min\{4,2 + k\}$.  That is, the iterates pick up a single order of an order of accuracy with each sweep up the solver, up to a maximal order based on the underlying quadrature rule.
\end{remark}

We formalize the statement of Rmk.\,\ref{rmk:consistency} in Thm.\,\ref{thm:consistency} but we first lay down the foundational ingredients for its proof.  As this method is based on the integral formulation of the differential equation, we begin with some lesser well known quadrature identities.  Define, for some generic function $f: \R^2 \rightarrow \R$,
\begin{align} 
\I\left[f^n, \alg f k\right] := \frac{\dt}{2} \left(f^n + \alg f k\right) + \frac{\dt^2}{12}\left({\dot f^n}-\alg {\dot f} k\right),
\end{align}
with the obvious notation $f^n := f(y^n, z^n)$ and $\alg f k := f(\alg y k, \alg z k)$.  Note that $\I$ is a fourth-order accurate quadrature rule, and therefore
\begin{align} 
\label{eqn:IzInt}
\I\left[z(t^n), z(t^{n+1})\right] &= \int_{t^n}^{t^{n+1}} z(t)\, \Dt + \O(\dt^5), \quad \text{and} \\
\label{eqn:IgInt}
\I\left[g(t^n), g(t^{n+1})\right] &= \int_{t^n}^{t^{n+1}} g(y(t),z(t))\, \Dt + \O(\dt^5),
\end{align}
assuming enough regularity in the underlying $y,z$, and $g$ functions that define \eqref{eq:vdp}.  (The constants in the big-$\O$ estimate do of course depend on $\eps > 0$.)  For the sake of readability, we have made the slight abuse of notation and are thinking of $g(t^n) := g(y(t^n),z(t^n))$.  Furthermore, observe that the defining equations for $\alg y k$, $\alg z k$, respectively, with $k > 0$ in the correction step can then be written as 
\begin{align}
\alg y {k+1} &= y^n + \I\left[z^n, \alg z k\right], \quad \text{and} \\
\alg z {k+1} &= z^n + \frac{\dt}{\eps} \left( \alg g {k+1} - \alg g k\right) - 
\frac{\dt^2}{2\eps} \left(\alg {\dot g} {k+1} - \alg {\dot g} {k} \right)
+ \frac 1 \eps \I\left[g^n, \alg g k\right],
\end{align}
with the understanding that $\dot z := \frac {g}{\eps}$, which is required to compute $\I\left[z^n, \alg z k\right]$.

As is customary in a consistency analysis, assume that $y^n$ and $z^n$ are the exact solutions evaluated at time $t^n$.  That is, we assume $y^n = y(t^n)$ and $z^n = z(t^n)$.  Define
\begin{equation}
\alg {\delta_y} k := \alg y k - y(t^{n+1}), \qquad \alg {\delta_z} k := \alg z k - z(t^{n+1}), \qquad \text{and} \qquad 
\alg \delta k := \|(\alg {\delta_y} k, \alg {\delta_z} k)\|.
\end{equation}
Note that
\begin{equation}
\label{eqn:Iz}
\begin{aligned} 
\left|\I\left[z(t^n),z(t^{n+1})\right] - \I\left[z(t^n), \alg z k\right]\, \right| 
&= \left| \frac{\dt}{2} \left(z( t^{n+1} ) - \alg z k\right) + \frac{\dt^2}{12}\left({\dot z(t^{n+1})}-\alg {\dot z} k\right) \right| \\
&\leq\frac{\dt}{2}\left| (z(t^{n+1})-\alg z k)\right| + \frac{\dt^2}{12\eps }\left|g(t^{n+1}) - \alg g k \right| \\
&\leq \frac{\dt}{2} \alg \delta k + \frac{\dt^2}{12\eps} L_g \alg \delta k,
\end{aligned}
\end{equation}
where $L_g$ is the Lipschitz constant for $g$, and similarly 
\begin{equation}
\label{eqn:Ig}
\left| \I\left[g(t^n), g(t^{n+1})\right] - \I\left[g(t^n), g^{[k+1]}\right] \right|
\leq \frac \dt 2 L_g \alg \delta k + \frac{\dt^2}{12 \eps} L_{\dot g} \alg \delta k,
\end{equation}
where $L_{\dot g}$ is the Lipschitz constant for $\dot{g}$.  Since the exact solution of the differential equation satisfies
\begin{equation} 
z(t^{n+1}) = z(t^n) + \frac 1 \eps \int_{t^n}^{t^{n+1}} g(y,z) \Dt, 
\end{equation}
we have
\begin{align*} 
\left| \alg {\delta_z} {k+1} \right| &= 
\left| z(t^n) + \frac{\dt}{\eps} \left( \alg g {k+1} - \alg g k\right) - 
\frac{\dt^2}{2\eps} \left(\alg {\dot g} {k+1} - \alg {\dot g} {k} \right)
+ \frac 1 \eps \I\left[g^n, \alg g k\right] - z(t^n) - \frac 1 \eps \int_{t^n}^{t^{n+1}} g(y,z) \Dt \right| \\
&\leq \frac{\dt}{\eps} \underbrace{\left| g^{[k+1]} - g^{[k]} \right| }_{\bf I} + 
\frac{\dt^2}{2 \eps} \underbrace{\left| \dot{g}^{[k+1]} - \dot{g}^{[k]} \right|}_{\bf II} +
\frac{1}{\eps} \underbrace{\left| \I\left[g^n, \alg g k\right]  - \int_{t^n}^{t^{n+1}} g(y,z) \Dt \right|}_{\bf III}.
\end{align*}
We estimate each of these terms separately:
\begin{equation}
|{\bf I}| = \left| g^{[k+1]} - g^{[k]} \right| \leq 
\left| g^{[k+1]} - g^{n+1} \right| + \left| g^{n+1} - g^{[k]} \right|
 \leq  L_g \left| \alg \delta {k+1} \right| + L_g \left| \alg \delta {k} \right|,
\end{equation}
and
\begin{align}
|{\bf II}| &= \left| \alg {\dot{g}}{{k+1}} - \alg {\dot{g}}{k} \right| =
\left| \dot{g}^{[k+1]} - \dot{g}^{n+1} + \dot{g}^{n+1} - \dot{g}^{[k]} \right|
\leq 
\frac{L_{\dot{g}}}\eps  \alg \delta{k+1} + \frac{L_{\dot{g}}}\eps \alg \delta {k} .
\end{align}
Finally, we make use of \eqref{eqn:Ig} and \eqref{eqn:IgInt} to estimate the third term in this inequality:
\begin{equation}
\begin{aligned}
\left| {\bf III} \right| &= \left| \I\left[g^n, \alg g k\right]  - \int_{t^n}^{t^{n+1}} g(y,z) \Dt \right|
\leq \left| \I\left[g^n, \alg g k\right]  - \I\left[ g^n, g^{n+1} \right] \right| +
\left| \I\left[ g^n, g^{n+1} \right] - \int_{t^n}^{t^{n+1}} g(y,z) \Dt \right| \\
&\leq \frac \dt 2 L_g \alg \delta k  + \frac{\dt^2}{12 \eps} L_{\dot g} \alg \delta k  + \O( \dt^5 ).
\end{aligned}
\end{equation}
All together, we have
\begin{equation}
\label{eqn:deltaz-estimate}
\begin{aligned}
\left| \alg {\delta_z} {k+1} \right| &\leq \frac{\dt}{\eps} \left| {\bf I} \right| + \frac{\dt^2}{2 \eps} \left| {\bf II} \right| + \frac1 \eps \left| {\bf III} \right|  \\
&\leq
\frac{L_g \dt}{\eps} \alg \delta {k+1} + \frac{L_g \dt}{\eps} \alg \delta {k} 
+ \frac{L_{\dot g}\dt^2}{2\eps^2} \alg \delta {k+1} + \frac{L_{\dot g}\dt^2}{2\eps^2} \alg \delta {k}
+\frac \dt {2 \eps} L_g \alg \delta k  + \frac{\dt^2}{12 \eps^2} L_{\dot g}  \alg \delta k  + \O( \dt^5 ) \\
&= \O(\dt \alg \delta {k+1}) + \O(\dt \alg \delta k) + \O(\dt^5).
\end{aligned}
\end{equation}
%
Note again that the constants in the $\O$-terms depend on $\eps$.  Similar results hold for $\alg {\delta_y} {k+1}$, which show that
\begin{equation}
\label{eqn:deltay-estimate}
\left| \alg {\delta_y} {k+1} \right| \leq \O(\dt \alg \delta {k+1}) + \O(\dt \alg \delta k) + \O(\dt^5).
\end{equation}
These results indicate that $\alg \delta {k+1}$ is one order (in $\dt$) better than $\alg \delta {k}$, until it reaches the maximum order of the quadrature rule.  We formalize this statement in the following theorem.
\begin{theorem}\label{thm:consistency}
The errors in the iterated approximations defined in Algorithm \ref{alg:mdvdp} satisfy $\alg \delta k = \O(\dt^{\min\{5,2+k\}+1})$, with any $k \in \Z_{\geq 0}$, and therefore when $k_{\max} \geq 2$, the method is fourth-order consistent.
\end{theorem}
  
\begin{proof}
The predictor is second-order accurate because $\alg y 0$ and $\alg z 0$ are computed by a second-order forward/backward Taylor method.  That is, $\alg \delta 0 = \O( \dt^3 )$.  Combining \eqref{eqn:deltaz-estimate} and \eqref{eqn:deltay-estimate} gives
\[
\alg \delta {k+1} = \O(\dt \alg \delta {k+1}) + \O(\dt \alg \delta k) + \O(\dt^5),
\]
which yields the desired result after applying induction on the number of iterates, $k$.

%
%
\end{proof}

\begin{remark}
From the analysis it is evident that once the quadrature operator $\I$ is replaced by another, higher-order quadrature, the method exhibits a higher overall order of accuracy.  This route opens the possibility to investigate even higher order semi-implicit multiderivative time integrators.
\end{remark}

\section{Asymptotic consistency}
\label{sec:ac}

Considering the fact that the algorithm should approximate a singularly perturbed equation, it is evident that the behavior of the algorithm in the limiting case $\eps \to 0$ is of utmost importance. Here, we investigate the asymptotic preserving (AP) property of the proposed method. Roughly speaking, an asymptotic preserving scheme means that the discretization found by sending $\eps \rightarrow 0$ but holding $\dt$ constant is a consistent discretization of the limit equation, Eqn. \eqref{eq:vdp_lim}, found by sending $\eps \to 0$ of the continuous problem.  Generic differential equation solvers do not typically have this property.

Formally, if $w_{{\dt},(\eps)}$ is a discretization of the stiff equations defined in \eqref{eq:vdp}, then there are two limits that can be computed.  We either send $\dt \to 0$ or we can send $\eps \to 0$, from which we send the other variable to zero.  On the one hand, if we first send $\dt \to 0$, then we end up with a (to be expected) numerical approximation $w_{(\eps)}$ of \eqref{eq:vdp}, which we understand relaxes to  $w_{(0)}$ as $\eps \to 0$.  On the other hand, if we instead first send $\eps \to 0$, then we end up with a discretization $w_{\Delta t, (0)}$, which may or may not converge to the limiting solution $w_{(0)}$ as $\dt \to 0$.  If it does, we say the numerical method has the \textit{asymptotic preserving property}.  This property is summarized in Figure \ref{fig:ap}.

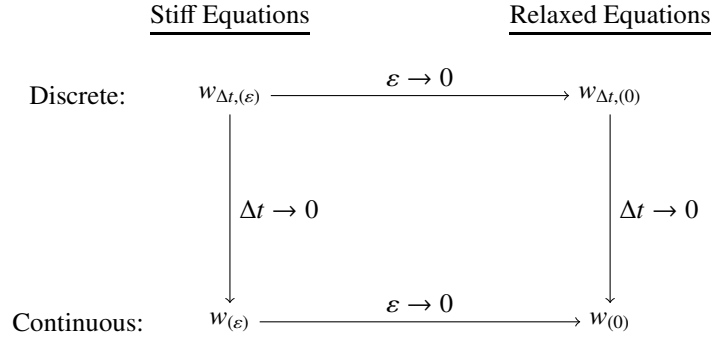
\begin{figure}
\begin{center}
	\begin{tikzpicture}[scale=1.0]
	\node (stiff) at (0,4) {\underline{Stiff Equations}};
	\node (relaxed) at (5,4) {\underline{Relaxed Equations}};
	\node (discrete) at (-2,3) {Discrete:};
	\node (continuous) at (-2,0) {Continuous:};
	\node (feps) at (0,0) {$w_{(\eps)}$};
	\node (f0)   at (5,0) {$w_{(0)}$};
	\node (neps) at (0,3) {$w_{{\dt},(\eps)}$};
	\node (n0)   at (5,3) {$w_{{\dt},(0)}$};
	\draw[->] (feps) -- (f0) node [auto,midway] {$\eps\to 0$};
	\draw[->] (neps) -- (n0) node [auto,midway] {$\eps\to 0$};
	\draw[->] (neps) -- (feps) node [auto,midway] {$\Delta t\to 0$};
	\draw[->] (n0) -- (f0) node [auto,midway] { $\Delta t\to 0$};
	\end{tikzpicture}
	\caption{Asymptotic preserving methods.  We say a method is asymptotic preserving if the limits in the above diagram commute with each other. That is, $\lim_{\dt \to 0} \lim_{\eps\to0} w_{{\dt},(\eps)} = \lim_{\eps\to0} \lim_{\Delta t\to 0} w_{{\dt}, (\eps)} = w_{(0)}$.  This property is not automatically preserved with any arbitrary, but consistent numerical method.
		\label{fig:ap}
	}
\end{center}
\end{figure}

We begin by showing the well-posedness of $\alg y 0$, $\alg z 0$, and the fact that these quantities possess Hilbert expansions. 
\begin{lemma}\label{la:hilbert}
 Assume, in addition to Postulate \ref{pos:lipschitz}, that $\partial_z g$ and $\partial_z \dot g$ are Lipschitz in the second argument, i.e., for all $y, z_1, z_2 \in \R$, we have
 \begin{equation}
  |\partial_z g(y,z_1) - \partial_z g(y,z_2)| \leq L_{\partial_z g} |z_1 -z_2|, \qquad 
  |\partial_z \dot g(y,z_1) - \partial_z \dot g(y,z_2)| \leq \frac{L_{\partial_z \dot g}}{\eps}|z_1 -z_2|,
 \end{equation}
and furthermore assume that all occurring derivatives of $g$ are uniformly bounded, and, in the spirit of \cite[Sec. VI.3]{HaiWan} that
 \begin{align}
\label{eq:hilbertconds}
 g(y^n,z^n) = \O(\eps\dt), \qquad \dot g(y^n,z^n) = \O\left(\frac{\dt}{\eps}\right).
 \end{align}
 If, in addition to these criteria, we assume $y^n$ and $z^n$ possess Hilbert expansions, then there exists a fixed $\eps_0 > 0$ and $\dt_0 > 0$, such that for all $0 < \eps < \eps_0$ and  $0 < \dt < \dt_0$, we have that $\alg y 0$ and $\alg z 0$ possess Hilbert expansions. 
 
\end{lemma}
\begin{remark}
 Under the assumption that there exists a Hilbert expansion, one can show that the identities in \eqref{eq:hilbertconds} hold with $\O(\eps)$. Behind this formulation is hence the implicit assumption that $\eps \ll \dt$. 
\end{remark}

\begin{proof}
 Due to the fact that $\alg y 0$ is computed explicitly, and $g(y^n, z^n) = \O(\eps\dt)$, it is evident that $\alg y 0$ possesses a Hilbert expansion, can hence be written as  
 \begin{align*} 
  \alg y 0 = \alg {\lims y 0} 0 +\eps \alg {\lims y 1} 0 + \O(\eps^2). 
 \end{align*}
 More challenging is showing that $\alg z 0$ has a Hilbert expansion, given that this is nonlinear and implicit.  Note that this term is supposed to be a zero of
 \begin{align}
  \label{eq:hilbertint1}
  F(\alg z 0) := \alg z 0 - \frac{\dt}{\eps} g(\alg y 0,\alg z 0) + \frac{\dt^2}{2\eps} {\dot g(\alg y 0,\alg z 0)} - z^n = 0.
 \end{align}
 As typically done, see \cite{HaiWan}, we apply Newton-Kantorovich's theorem to this function $F$. Direct computation gives
 \begin{align*}
  F'(z) = 1 - \frac{\dt}{\eps} \partial_z g(\alg y 0, z)  + \frac{\dt^2}{2\eps} \partial_z \dot g(\alg y 0,z),
 \end{align*}
  and hence 
 \begin{align*}
  |F'(z_1) - F'(z_2)| \leq \left( \frac{\dt}{\eps} L_{\partial_z g} + \frac{\dt^2}{2\eps^2} L_{\partial_z \dot g} \right) |z_1-z_2|.
 \end{align*}
 Furthermore, observe that 
 \begin{align*} 
  \left|g(\alg y 0, z^n)\right| = \left|g\left(y^n + \O(\dt),z^n\right)\right| \leq |g(y^n, z^n)| + \O(\dt) \leq M_g\dt,
 \end{align*}
 for some constant $M_g$, because $g(y^n,z^n) = \O(\eps\dt)$, $\partial_y g$ is bounded and there is some upper bound on $\eps$. 
 Because of our assumption on bounded derivatives of $g$, we also have
 \begin{align*} 
  \partial_{zz} g(\alg y 0,z^n) g(\alg y 0,z^n) \leq M_1\dt
 \end{align*}
 for some $M_1 > 0$. 
 
 Now, consider Newton's method applied to $F$, with initial point $z^n$.   Choose an $M_2$ such that  
 \begin{align*} 
  \|\partial_{yz}g\|_{\infty}|z^n| + \|\partial_y g\|_{\infty}  \leq M_2.
 \end{align*}
 Then, taking into account Postulate \ref{pos:dzg}, we have
 \begin{align*}
  F'(z^n) &=    1 - \frac{\dt}{\eps} {\partial_z g}(\alg y 0,z^n) + \frac{\dt^2}{2\eps} \partial_z \dot g (\alg y 0,z^n) \\
          &\geq 1 + \frac{\dt}{\eps} c + \frac{\dt^2}{2\eps} \left(\partial_{yz} g(\alg y 0,z^n) z^n + \partial_y g(\alg y 0,z^n)\right) + 
          \frac{\dt^2}{2\eps^2} \left(\partial_{zz} g(\alg y 0,z^n) g(\alg y 0,z^n) + (\partial_z g(\alg y 0,z^n))^2\right) \\
          &\geq 1 + \frac{\dt}{\eps} c + \frac{\dt^2}{2\eps^2} c^2 + \frac{\dt^2}{2\eps} \left(\partial_{yz} g(\alg y 0,z^n) z^n + \partial_y g(\alg y 0,z^n)\right)
          + \frac{\dt^2}{2\eps^2} \left(\partial_{zz} g(\alg y 0,z^n) g(\alg y 0,z^n)\right) \\
          &\geq 1 + \frac{\dt}{\eps}\left(c-\frac{\dt}{2}M_2\right) + \frac{\dt^2}{2\eps^2}\left(c^2-\dt M_1\right).
 \end{align*}
 Choosing $\dt_0$ small enough (independently of $\eps$!) so that the expressions in brackets are strictly positive (larger than $\alpha >0$ say) for any $\dt < \dt_0$,  yields that $F'(z^n) \neq 0$ and that 
 \begin{align*}
  F'(z^n) \geq 1 + \frac{\dt}{\eps} \alpha + \frac{\dt^2}{2\eps^2}\alpha.
 \end{align*}

 A similar computation, taking into account $g(y^n,z^n) = \O(\dt)$ and $\dot g(y^n,z^n) = \O\left(\frac{\dt}{\eps}\right)$, yields 
 \begin{align*}
  |F(z^n)| &= \left|z^n - \frac{\dt}{\eps} g(\alg y 0,z^n) + \frac{\dt^2}{2\eps} \dot g(\alg y 0,z^n) - z^n\right| \\
  &\leq M_g\frac{\dt^2}{\eps} + M_{\dot g} \frac{\dt^3}{2\eps^2},
 \end{align*}
 where $M_{\dot g}$ is defined similarly to $M_g$.

 The first Newton step would thus have step width 
 \begin{align*}
  \frac{F(z^n)}{F'(z^n)} \leq \dt \frac{M_g\frac{\dt}{\eps} + M_{\dot g} \frac{\dt^2}{2\eps^2}}{1 + \frac{\dt}{\eps} \alpha + \frac{\dt^2}{2\eps^2}\alpha}.
 \end{align*}
 This expression can be bounded by $\dt$ times a constant $H$ that does not depend on $\eps$ and $\dt$, i.e., 
 \begin{align*}
  \frac{F(z^n)}{F'(z^n)} \leq H \dt.
 \end{align*}
 Hence, there holds 
 \begin{align*} 
  \left( \frac{\dt}{\eps} L_{\partial_z g} + \frac{\dt^2}{2\eps^2} L_{\partial_z \dot g} \right) |F'(z^n)^{-1}| \left|\frac{F(z^n)}{F'(z^n)} \right| 
  \leq \frac{\frac{\dt}{\eps} L_{\partial_z g} + \frac{\dt^2}{2\eps^2} L_{\partial_z \dot g}}{1 + \frac{\dt}{\eps} \alpha + \frac{\dt^2}{2\eps^2}\alpha}H \dt.
 \end{align*}
 Also this can be bounded by some constant times $\dt$, hence, choosing $\dt$ sufficiently small makes the expression smaller than one half, and the Newton-Kantorovich theorem can be used. Not only does this imply that $\alg z 0 - z^n = \O(1)$, it also implies that $\alg z 0$ has a Hilbert expansion, because we can repeat the argument for every Newton step.
 \end{proof}

\begin{remark}
 Please note that statement and proof can be extended to the full method. 
\end{remark}


To continue, we show the AP-property for the forward/backward starting phase. 
\begin{lemma}\label{la:apeuler}
 Assume that $\alg y 0$ and $\alg z 0$ possess Hilbert expansions.  That is, we are operating under the assumptions presented in Lemma \ref{la:hilbert}.
 Then, $\alg y 0$ and $\alg z 0$ are also well-prepared in the sense of Def. \ref{def:wellprep}, i.e., there holds 
 \begin{align*}
  \alg {\lims g 0} 0 := g\left(\alg {\lims y 0} 0, \alg {\lims z 0} 0\right) = 0, \qquad 
  \nabla {\alg {\lims g 0} 0}  \cdot \left(\begin{matrix} \lims {\alg z 0} 0 \\ \nabla \alg {\lims g 0} 0 \cdot \left(\begin{matrix} \lims {\alg y 0} 1 \\ \lims {\alg z 0} 1 \end{matrix}\right) \end{matrix}\right)=0.
 \end{align*}
\end{lemma}
\begin{proof}
 The proof starts by considering $\alg z 0$, given by 
 \begin{align*}
  \alg z 0 &:= z^n + \frac{\dt}{\eps} \alg g 0 - \frac{\dt^2}{2\eps} \nabla {\alg g 0}   \cdot \left(\begin{matrix} \alg z 0 \\ \frac 1 \eps \alg g 0 \end{matrix} \right). 
 \end{align*}
 Inserting a Hilbert expansion for all occurring quantities reveals the fact that the highest order is $\O(\eps^{-2})$, with corresponding equation being given by
 \begin{align*}
  \partial_z \alg {\lims g 0} 0  \alg {\lims g 0} 0 = 0.
 \end{align*}
 Due to our assumption on $\partial_z g$, see Postulate \ref{pos:dzg}, there follows $\alg {\lims g 0} 0 = 0$. 
 
 Now, to $\O(\eps^{-1})$, the limiting equations are (be aware that the term scaled with $\eps^{-2}$ also contributes, see Eqn. \eqref{eq:taylor_on_g}) are given by
 \begin{align*}
  \dt \alg {\lims g 0} 0 - \frac{\dt^2}{2} \nabla {\lims {\alg g 0} 0} \cdot \left(\begin{matrix} \lims {\alg z 0} 0 \\ \nabla \lims {\alg g 0} 0 \cdot \left(\begin{matrix} \lims {\alg y 0} 1 \\ \lims {\alg z 0} 1 \end{matrix} \right)\end{matrix}\right)
  - \frac{\dt^2}{2} \nabla (\partial_z \alg {\lims g 0} 0) \cdot \left(\begin{matrix} \lims {\alg y 0} 1 \\ \lims {\alg z 0} 1 \end{matrix}\right) \cdot \lims {\alg g 0} 0 = 0.
 \end{align*}
 Exploiting the fact that $\lims {\alg g 0} 0 = 0$ yields the claim. 
\end{proof}

\begin{theorem}[The algorithm is AP]
 Assume that the discrete solution at time level $t = 0$ is well-prepared in the sense of Def.~\ref{def:wellprep}. 
 Assume furthermore that the discrete solution possesses a Hilbert expansion. 
 Then, there holds for all times $t^n$ that 
 \begin{align*}
  g(\lims {y^n} 0,\lims {z^n} 0) = 0
 \end{align*}
 and 
 \begin{align*}
  \nabla g(\lims {y^n} 0, \lims {z^n} 0) \cdot \left(\begin{matrix} \lims {z^n} 0 \\ \nabla g(\lims {y^n} 0, \lims {z^n} 0) \cdot \left(\begin{matrix} \lims {y^n} 1 \\ \lims {z^n} 1 \end{matrix}\right) \end{matrix}\right) = 0.
 \end{align*}
 This implies that the method is asymptotic preserving. 
\end{theorem}
\begin{proof}
 The proof is very similar to the one of La.~\ref{la:apeuler} and is hence omitted. Note that due to the $\eps-$dependency of the `explicit' terms $\alg {g} k$ and $\alg {\dot g} k$, terms as in Eqn. \eqref{eq:wellprep} at previous time/stage instances do show up. This, contrarily to La.~\ref{la:apeuler}, necessitates the need for well-prepared initial conditions as in Def.~\ref{def:wellprep}. 
\end{proof}

\section{Asymptotic accuracy}
\label{sec:aa}
From a practical point of view, it is not only of interest whether the method is asymptotically \emph{consistent}, but also to what orders the consistency is. This is a more delicate issue than pure consistency; in particular for IMEX Runge-Kutta methods, this leads to rather unwanted results including the stage order of the Runge-Kutta method being a limiting factor, see \cite{Bo07,HaiWan}. In order to investigate this, let us consider van der Pol equation, being in the form \eqref{eq:vdp} with $g$ given by 
\begin{align*}
 g(y,z) = (1-y^2)z - y. 
\end{align*}
We put our initial conditions as
\begin{align*}
 y(0) = 2, \qquad z(0) = -\frac 2 3  + \frac{10}{81} \eps - \frac{292}{2187} \eps^2, 
\end{align*}
which is a frequent choice in literature, see, e.g., \cite{Bos09}. Note that these initial conditions are well-prepared in the sense of Def. \ref{def:wellprep}.

In Fig. \ref{fig:imexvdp}, we plot convergence results for the method presented in Alg. \ref{alg:mdvdp}. On the top-left, we chose $k_{\max} = 0$ (this means that $y^{n+1} = \alg y 0$, similarly for $z$, i.e., only the predictor is taken into account). On the top-right, $k_{\max}$ is set to 2, which is the minimal number of iterations required to produce a fourth-order scheme. 
We observe that the second-order scheme (which is our second-order base IMEX Taylor solver) exhibits no order reduction.  This means that there is second-order convergence uniformly in $\eps$.  The fourth-order scheme shows severe order reduction. On the bottom of Fig.\ \ref{fig:imexvdp}, we increase $k_{\max}$ to 20 and 100, respectively. It is clearly visible that this enhances convergence.  For example, with $k_{\max} = 100$ we observe no order reduction. The reason for this behaviour is that under the assumption that $\left(\alg y k, \alg z k\right)$ converges as $k \rightarrow \infty$,
the result of Alg. \ref{alg:mdvdp} is equal to the fourth-order quadrature rule:
\begin{align*}
 y^{n+1} = y^n + \frac {\dt}{2} \left(z^n + z^{n+1}\right) + \frac{\dt^2}{12\eps}\left(g^n - g^{n+1}\right), \\
 z^{n+1} = z^n + \frac{\dt}{2\eps} \left(g^n + g^{n+1}\right) + \frac{\dt^2}{12 \eps} \left(\dot g^n - \dot g^{n+1}\right), 
\end{align*}
which is apparently less sensitive to order reduction.

\begin{figure}[h]
	\begin{center}
		\begin{tikzpicture}[scale=0.64]
		\begin{loglogaxis}[ymax=1e-2,ymin=1e-16,xlabel={Size of $\dt$},ylabel={Error $e_{\Delta t}$},grid=major,legend style={at={(0.02,0.98)},anchor=north west,font=\footnotesize,rounded corners=2pt}]
		\addplot                              table[skip first n=1, x index = 0, y index =  1] {convergence_table_Taylor2.dat};
		\addplot[mark=triangle*,color=green]  table[skip first n=1, x index = 0, y index =  3] {convergence_table_Taylor2.dat};
		\addplot                              table[skip first n=1, x index = 0, y index =  5] {convergence_table_Taylor2.dat};
		\addplot                              table[skip first n=1, x index = 0, y index =  7] {convergence_table_Taylor2.dat};
		\addplot                              table[skip first n=1, x index = 0, y index =  9] {convergence_table_Taylor2.dat};
		\addplot                              table[skip first n=1, x index = 0, y index = 11] {convergence_table_Taylor2.dat};
		\end{loglogaxis}
		\end{tikzpicture}
		\begin{tikzpicture}[scale=0.64]
		\begin{loglogaxis}[ymax=1e-2,ymin=1e-16,xlabel={Size of $\dt$},ylabel={Error $e_{\Delta t}$},grid=major,legend style={at={(1.12,0.98)},anchor=north west,font=\footnotesize,rounded corners=2pt}]
		\addplot                              table[skip first n=1, x index = 0, y index =  1] {convergence_table.dat};
		\addplot[mark=triangle*,color=green]  table[skip first n=1, x index = 0, y index =  3] {convergence_table.dat};
		\addplot                              table[skip first n=1, x index = 0, y index =  5] {convergence_table.dat};
		\addplot                              table[skip first n=1, x index = 0, y index =  7] {convergence_table.dat};
		\addplot                              table[skip first n=1, x index = 0, y index =  9] {convergence_table.dat};
		\addplot                              table[skip first n=1, x index = 0, y index = 11] {convergence_table.dat};
		\legend{
 			$\eps = 1.0 \times 10^{-1}$,
 			$\eps = 1.0 \times 10^{-2}$,
 			$\eps = 1.0 \times 10^{-3}$,
 			$\eps = 1.0 \times 10^{-4}$,
 			$\eps = 1.0 \times 10^{-5}$,
 			$\eps = 1.0 \times 10^{-6}$,
 		}
		\end{loglogaxis}
		\end{tikzpicture}
		\\
		\begin{tikzpicture}[scale=0.64]
		\begin{loglogaxis}[ymax=1e-2,ymin=1e-16,xlabel={Size of $\dt$},ylabel={Error $e_{\Delta t}$},grid=major,legend style={at={(1.12,0.98)},anchor=north west,font=\footnotesize,rounded corners=2pt}]
		\addplot                              table[skip first n=1, x index = 0, y index =  1] {convergence_table_kmax20.dat};
		\addplot[mark=triangle*,color=green]  table[skip first n=1, x index = 0, y index =  3] {convergence_table_kmax20.dat};
		\addplot                              table[skip first n=1, x index = 0, y index =  5] {convergence_table_kmax20.dat};
		\addplot                              table[skip first n=1, x index = 0, y index =  7] {convergence_table_kmax20.dat};
		\addplot                              table[skip first n=1, x index = 0, y index =  9] {convergence_table_kmax20.dat};
		\addplot                              table[skip first n=1, x index = 0, y index = 11] {convergence_table_kmax20.dat};
		\end{loglogaxis}
		\end{tikzpicture}		
		\begin{tikzpicture}[scale=0.64]
		\begin{loglogaxis}[ymax=1e-2,ymin=1e-16,xlabel={Size of $\dt$},ylabel={Error $e_{\Delta t}$},grid=major,legend style={at={(1.12,0.0)},anchor=south west,font=\footnotesize,rounded corners=2pt}]
		\addplot                              table[skip first n=1, x index = 0, y index =  1] {convergence_table_kmax100.dat};
		\addplot[mark=triangle*,color=green]  table[skip first n=1, x index = 0, y index =  3] {convergence_table_kmax100.dat};
		\addplot                              table[skip first n=1, x index = 0, y index =  5] {convergence_table_kmax100.dat};
		\addplot                              table[skip first n=1, x index = 0, y index =  7] {convergence_table_kmax100.dat};
		\addplot                              table[skip first n=1, x index = 0, y index =  9] {convergence_table_kmax100.dat};
		\addplot                              table[skip first n=1, x index = 0, y index = 11] {convergence_table_kmax100.dat};
 		\legend{
 			$\eps = 1.0 \times 10^{-1}$,
 			$\eps = 1.0 \times 10^{-2}$,
 			$\eps = 1.0 \times 10^{-3}$,
 			$\eps = 1.0 \times 10^{-4}$,
 			$\eps = 1.0 \times 10^{-5}$,
 			$\eps = 1.0 \times 10^{-6}$,
 		}
		\end{loglogaxis}
		\end{tikzpicture}		
		\caption{Convergence results for van der Pol equation with different values of $\eps$. Top left: $k_{\max} = 0$, which amounts to only taking the predictor. Top right: $k_{\max} = 2$, with amounts to taking the full fourth-order scheme. Bottom: $k_{\max}=20$ and $k_{\max}=100$, respectively. Error measure is defined as the Euclidean norm of the $y$ and $z$ error at end time $\Tend = 0.5$.}\label{fig:imexvdp}
	\end{center}
\end{figure}

To investigate the loss of asymptotic order of accuracy numerically in some closer detail, we define $\delta(\dt;\eps)$ to be the Euclidean norm of the error in $y$ and $z$ at end time $\Tend$ for a given $\dt$ and a given $\eps$, i.e., 
\begin{align*}
 \delta(\dt;\eps) := \sqrt{\left(y^N-y(\Tend)\right)^2 + \left(z^N-z(\Tend)\right)^2}. 
\end{align*}
(Note that $\delta$ does of course also depend on $k_{\max}$, which we have not made explicit.) As for the solution, a Hilbert expansion of $\delta$ in terms of $\eps$ is assumed, so 
\begin{align}
 \label{eq:deltaexp}
 \delta(\dt,\eps) = \delta_0(\dt) + \eps \delta_1(\dt) + \eps^2 \delta_2(\dt) + \O(\eps^3). 
\end{align}
We approximate $\delta_0(\dt)$ and $\delta_1(\dt)$ numerically through 
\begin{align}
\label{eq:deltaexpnum}
 \delta_0 \approx \frac{\delta(\dt; \alpha \eps) - \alpha \delta(\dt;\eps)}{1-\alpha}, \qquad
 \eps \delta_1 \approx \omega_1 \delta(\dt; \eps) + \omega_2 \delta(\dt; \alpha \eps) + \omega_3 \delta(\dt; \alpha^2 \eps)
\end{align}
where we choose the rather arbitrary values $\alpha = \frac 5 6$ and $\eps = \alpha^2 \cdot 10^{-5}$.  (In our numerical testing, we verified that the results obtained are not influenced to any significant accuracy by this choice of $\eps$.)
The weights $\omega_i$ are chosen so that 
\begin{align*}
 \omega_1 + \omega_2 + \omega_3 = 0, \quad \omega_1 + \alpha \omega_2 + \alpha^2 \omega_3 = 1, \quad \omega_1 + \omega_2 \alpha^2 + \omega_3 \alpha^4 = 0.
\end{align*}
These conditions on the $\omega_i$ come out naturally after inserting the expansion \eqref{eq:deltaexp} into \eqref{eq:deltaexpnum}.

For the same test case as above (i.e., van der Pol's problem with $k_{max} = 0, 2, 20, 100$, respectively) we plot $\delta_0(\dt)$ and $\delta_1(\dt)$ in Fig.\ \ref{fig:imexvdpasympt}.  We remark that it is the contribution of $\delta_1(\dt)$ that is responsible for the degredation in the order of the solver.  It can be seen that the slope of $\delta_1(\dt)$ increases as $k_{\max}$ increases, until machine accuracy issues occur.

\begin{figure}[h]
	\begin{center}
	\begin{tikzpicture}[scale=0.64]
		\begin{loglogaxis}[ymax=1e-0,ymin=1e-16,xlabel={Size of $\dt$},ylabel={$\delta(\dt) = \delta_0(\dt) + \eps \delta_1(\dt) + \O(\eps^2)$},grid=major,legend style={at={(0.02,0.98)},anchor=north west,font=\footnotesize,rounded corners=2pt}]
		\addplot                              table[x index = 0, y index = 1] {convergence_table_asymptotic_Taylor2.dat};
		\addplot                              table[x index = 0, y index = 2] {convergence_table_asymptotic_Taylor2.dat};
		\legend{$\delta_0(\dt)$, $\delta_1(\dt)$}
		\end{loglogaxis}
	\end{tikzpicture}	\begin{tikzpicture}[scale=0.64]
		\begin{loglogaxis}[ymax=1e-2,ymin=1e-16,xlabel={Size of $\dt$},grid=major,legend style={at={(0.02,0.98)},anchor=north west,font=\footnotesize,rounded corners=2pt}]
		\addplot                              table[x index = 0, y index = 1] {convergence_table_asymptotic.dat};
		\addplot                              table[x index = 0, y index = 2] {convergence_table_asymptotic.dat};
		\legend{$\delta_0(\dt)$, $\delta_1(\dt)$}
		\end{loglogaxis}
	\end{tikzpicture}
	\\
	\begin{tikzpicture}[scale=0.64]
		\begin{loglogaxis}[ymax=1e-2,ymin=1e-16,xlabel={Size of $\dt$},ylabel={$\delta(\dt) = \delta_0(\dt) + \eps \delta_1(\dt) + \O(\eps^2)$},grid=major,legend style={at={(0.02,0.98)},anchor=north west,font=\footnotesize,rounded corners=2pt}]
		\addplot                              table[x index = 0, y index = 1] {convergence_table_asymptotic_kmax20.dat};
		\addplot                              table[x index = 0, y index = 2] {convergence_table_asymptotic_kmax20.dat};
		\legend{$\delta_0(\dt)$, $\delta_1(\dt)$}
		\end{loglogaxis}
	\end{tikzpicture}	
		\begin{tikzpicture}[scale=0.64]
		\begin{loglogaxis}[ymax=1e-2,ymin=1e-16,xlabel={Size of $\dt$},grid=major,legend style={at={(0.02,0.98)},anchor=north west,font=\footnotesize,rounded corners=2pt}]
		\addplot                              table[x index = 0, y index = 1] {convergence_table_asymptotic_kmax100.dat};
		\addplot                              table[x index = 0, y index = 2] {convergence_table_asymptotic_kmax100.dat};
		\legend{$\delta_0(\dt)$, $\delta_1(\dt)$}
		\end{loglogaxis}
	\end{tikzpicture}	
	\caption{Asymptotic convergence results for van der Pol equation. Top left: $k_{\max} = 0$, which amounts to only taking the predictor. Top right: $k_{\max} = 2$, which is the minimum required number of iterations to obtain the the full fourth-order scheme. Bottom: $k_{\max}=20$ and $k_{\max}=100$, respectively. Error measure is defined as the Euclidean norm of the $y$ and $z$ error at end time $\Tend = 0.5$.}\label{fig:imexvdpasympt}\label{fig:imexbdf2}
	\end{center}
\end{figure}
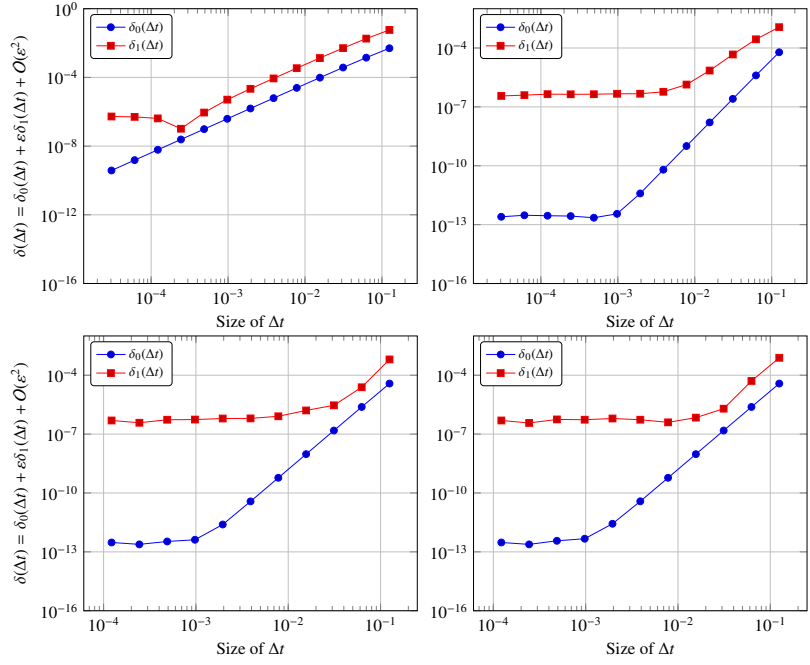

\section{Extensions and stability}
\label{sec:stability}
\subsection{Extending the method to arbitrary splittings}
 
So far, we have discussed the method for equations of type Eqn.\, \eqref{eq:vdp}, and we have proposed a solver for that class of equations in Algorithm \ref{alg:mdvdp}.  The method developed in this work has a natural extension to a larger class of ODEs, which we now point out. Consider, for example a system of ODEs with a generic additive right hand side:
\begin{align}
 \label{eq:ode}
 w'(t) = \Phi(w) =: \Phie(w) + \Phii(w),
\end{align}
with $w: \R^{\geq 0} \rightarrow \R^n$ and $\Phie, \Phii: \R^n \rightarrow \R^n$ being a splitting of the right hand side function.  It is assumed that $\Phii(w)$ contains ``stiff" terms, and $\Phie(w)$ contains non-stiff terms, so $\Phii(w)$ should be treated implicitly, while $\Phie(w)$ can be treated explicitly in order to speed up the computations. Note that Eqn.\ \eqref{eq:vdp} is of this form, with $\Phii(w) = (0, \frac1 \eps g(w))$ and $\Phie(w) = (z,0)$.  The choice of a suitable splitting is a subtle issue.  We refer the reader to \cite{SKN15} for other splittings. 

We extend Algorithm \ref{alg:mdvdp} in the following fashion to address arbitrary splittings. One of the chief goals is to retain the implicit-explicit (IMEX) type flavor of the underlying ODE and reach higher orders of accuracy all the while keeping the implicit solves as simple as possible. 
Note that the total time derivative of each piece in the right hand side function is given by
\begin{align*} 
 \dot \Phii(w) = \Phii'(w) \left(\Phii(w) + \Phie(w)\right), \qquad \dot \Phie(w) = \Phie'(w) \left(\Phii(w) + \Phie(w)\right).
\end{align*}

\begin{algorithm}\label{alg:mdode}
Consider a differential equation with a split right hand side given by Eqn.\,\eqref{eq:ode}.  To advance the solution from time level $t = t^n$ to $t = t^{n+1}$ we perform the following predictor-corrector strategy:
\begin{enumerate}
	\item \textbf{Predict.}  Given the solution $w^n$ at time level $t = t^n$, we first compute an approximation to $w^{[0]} \approx w^{n+1}$ via
	 \begin{align} 
	 \label{eq:predict}
	\alg w 0 &:= w^n + \dt \left(\Phii(\alg w 0) + \Phie( w^n ) \right) + \frac{\dt^2}{2} \left({\dot{\Phie}}(w^n) - {\dot{\Phii}}(\alg w 0) \right).
	\end{align}
	That is, we perform a forward Taylor expansion on $\Phie$ and a backward Taylor expansion on $\Phii$ and integrate the results.  This   produces a second-order accurate predictor.

	\item \textbf{Correct.} Based on this initial step, for each $0 \leq k \leq k_{\max}-1$ solve
	\begin{align*}
	\alg w {k+1} &:= w^n + {\dt} \left( \alg {\Phii} {k+1} - \alg {\Phii} k\right) - 
	\frac{\dt^2}{2} \left(\alg {\dot {\Phii}} {k+1} - \alg {\dot {\Phii}} {k} \right)
	+ \frac{\dt}{2} \left( \Phi^n + \alg \Phi k\right) + \frac{\dt^2}{12} \left(\alg {\dot {\Phi}^n -\dot \Phi} k \right)
	\end{align*}
	for $\alg w {k+1}$.

	\item \textbf{Update.}  Set $w^{n+1} := \alg w {k_{\max}}$.

\end{enumerate}


\end{algorithm}

Note that the intermediate iterates need not be stored, and therefore this algorithm needs only the solution at a total of two time levels.  This is advantageous when compared to any multistep method, where the solution at each stage needs to be stored, as well as most Runge-Kutta methods (save the methods of the low-storage variety).

\begin{remark} 
	Algorithm \ref{alg:mdode} is an extension of Algorithm \ref{alg:mdvdp}.  That is, 
 with $w = (y,z)$, $\Phie(w) = (z, 0)$, and $\Phii(w)= \left(0, \frac g \eps\right)$, Algorithm \ref{alg:mdode} reduces to Algorithm \ref{alg:mdvdp}. 
\end{remark}

We now present the results for this problem on some classical test cases from the literature.

\subsection{Kaps Problem}

A problem that is not of the form defined in \eqref{eq:vdp} is the so-called Kaps test problem \cite{TC10}
\begin{alignat*}{2}
	y' &= -2y + \frac{1}{\eps} ( z^2 - y ),&\qquad& y(0) = 1, \\
	z' &= y - z( 1 + z ), &\qquad&  z(0) = 1, 
\end{alignat*}
with exact solution $w := (y,z) =  (e^{-2t}, e^{-t} )$ for any $\eps > 0$.

We use the most straightforward splitting on this problem given by grouping all of the terms containing $\eps$ and putting them into the implicit piece of the right hand side:
\begin{align*} 
 \Phie(w) := \left(-2 y, y - z(1+y)\right)^T, \qquad \Phii(w) := \frac 1 \eps \left(z^2-y, 0 \right)^T.
\end{align*}
We present numerical results in Fig. \ref{fig:imexkaps}.  These results echo the findings of the previous section: 
\begin{itemize}
 \item The second-order scheme does not exhibit order degradation. 
 \item For low $k_{\max}$, we observe order degradation.
 \item For $k_{\max} \rightarrow \infty$, the observed order degradation vanishes. 
\end{itemize}

We thus conclude that the algorithm is capable of also computing solutions to equations that are not given in form \eqref{eq:vdp}. This is very important, in particular with respect to an extension of the scheme to singularly perturbed PDEs, where the semi-discretized systems are rarely in the form defined in \eqref{eq:vdp}. 

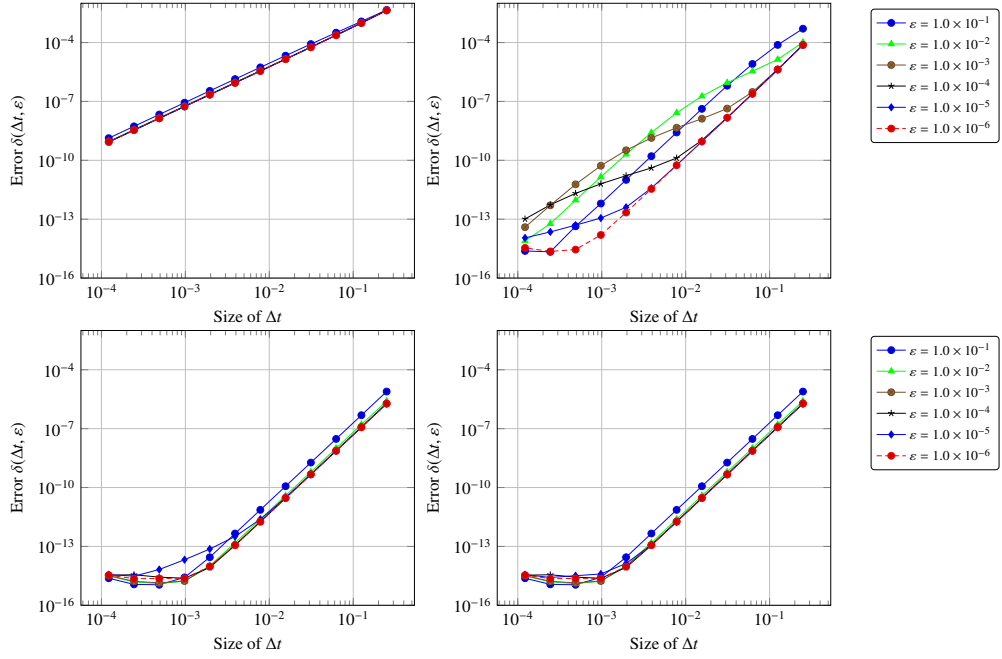
\begin{figure}[h]
	\begin{center}
		\begin{tikzpicture}[scale=0.64]
		\begin{loglogaxis}[ymax=1e-2,ymin=1e-16,xlabel={Size of $\dt$},ylabel={Error $\delta(\dt,\eps)$},grid=major,legend style={at={(0.02,0.98)},anchor=north west,font=\footnotesize,rounded corners=2pt}]
		\addplot                              table[skip first n=1, x index = 0, y index =  1] {convergence_table_kmax0_kaps.dat};
		\addplot[mark=triangle*,color=green]  table[skip first n=1, x index = 0, y index =  3] {convergence_table_kmax0_kaps.dat};
		\addplot                              table[skip first n=1, x index = 0, y index =  5] {convergence_table_kmax0_kaps.dat};
		\addplot                              table[skip first n=1, x index = 0, y index =  7] {convergence_table_kmax0_kaps.dat};
		\addplot                              table[skip first n=1, x index = 0, y index =  9] {convergence_table_kmax0_kaps.dat};
		\addplot                              table[skip first n=1, x index = 0, y index = 11] {convergence_table_kmax0_kaps.dat};
		\end{loglogaxis}
		\end{tikzpicture}
		\begin{tikzpicture}[scale=0.64]
		\begin{loglogaxis}[ymax=1e-2,ymin=1e-16,xlabel={Size of $\dt$},ylabel={Error $\delta(\dt,\eps)$},grid=major,legend style={at={(1.12,0.98)},anchor=north west,font=\footnotesize,rounded corners=2pt}]
		\addplot                              table[skip first n=1, x index = 0, y index =  1] {convergence_table_kmax2_kaps.dat};
		\addplot[mark=triangle*,color=green]  table[skip first n=1, x index = 0, y index =  3] {convergence_table_kmax2_kaps.dat};
		\addplot                              table[skip first n=1, x index = 0, y index =  5] {convergence_table_kmax2_kaps.dat};
		\addplot                              table[skip first n=1, x index = 0, y index =  7] {convergence_table_kmax2_kaps.dat};
		\addplot                              table[skip first n=1, x index = 0, y index =  9] {convergence_table_kmax2_kaps.dat};
		\addplot                              table[skip first n=1, x index = 0, y index = 11] {convergence_table_kmax2_kaps.dat};
        \legend{
 			$\eps = 1.0 \times 10^{-1}$,
 			$\eps = 1.0 \times 10^{-2}$,
 			$\eps = 1.0 \times 10^{-3}$,
 			$\eps = 1.0 \times 10^{-4}$,
 			$\eps = 1.0 \times 10^{-5}$,
 			$\eps = 1.0 \times 10^{-6}$,
 		}
		\end{loglogaxis}
		\end{tikzpicture}
		\\
		\begin{tikzpicture}[scale=0.64]
		\begin{loglogaxis}[ymax=1e-2,ymin=1e-16,xlabel={Size of $\dt$},ylabel={Error $\delta(\dt,\eps)$},grid=major,legend style={at={(1.12,0.98)},anchor=north west,font=\footnotesize,rounded corners=2pt}]
		\addplot                              table[skip first n=1, x index = 0, y index =  1] {convergence_table_kmax20_kaps.dat};
		\addplot[mark=triangle*,color=green]  table[skip first n=1, x index = 0, y index =  3] {convergence_table_kmax20_kaps.dat};
		\addplot                              table[skip first n=1, x index = 0, y index =  5] {convergence_table_kmax20_kaps.dat};
		\addplot                              table[skip first n=1, x index = 0, y index =  7] {convergence_table_kmax20_kaps.dat};
		\addplot                              table[skip first n=1, x index = 0, y index =  9] {convergence_table_kmax20_kaps.dat};
		\addplot                              table[skip first n=1, x index = 0, y index = 11] {convergence_table_kmax20_kaps.dat};
		\end{loglogaxis}
		\end{tikzpicture}		
		\begin{tikzpicture}[scale=0.64]
		\begin{loglogaxis}[ymax=1e-2,ymin=1e-16,xlabel={Size of $\dt$},ylabel={Error $\delta(\dt,\eps)$},grid=major,legend style={at={(1.12,0.98)},anchor=north west,font=\footnotesize,rounded corners=2pt}]
		\addplot                              table[skip first n=1, x index = 0, y index =  1] {convergence_table_kmax100_kaps.dat};
		\addplot[mark=triangle*,color=green]  table[skip first n=1, x index = 0, y index =  3] {convergence_table_kmax100_kaps.dat};
		\addplot                              table[skip first n=1, x index = 0, y index =  5] {convergence_table_kmax100_kaps.dat};
		\addplot                              table[skip first n=1, x index = 0, y index =  7] {convergence_table_kmax100_kaps.dat};
		\addplot                              table[skip first n=1, x index = 0, y index =  9] {convergence_table_kmax100_kaps.dat};
		\addplot                              table[skip first n=1, x index = 0, y index = 11] {convergence_table_kmax100_kaps.dat};
 		\legend{
 			$\eps = 1.0 \times 10^{-1}$,
 			$\eps = 1.0 \times 10^{-2}$,
 			$\eps = 1.0 \times 10^{-3}$,
 			$\eps = 1.0 \times 10^{-4}$,
 			$\eps = 1.0 \times 10^{-5}$,
 			$\eps = 1.0 \times 10^{-6}$,
 		}
		\end{loglogaxis}
		\end{tikzpicture}		
		\caption{Convergence results for Kaps problem with different values of $\eps$. Top left: $k_{\max} = 0$, which amounts to only taking the predictor. Top right: $k_{\max} = 2$, with amounts to taking the full fourth-order scheme. Bottom: $k_{\max}=20$ and $k_{\max}=100$, respectively. Error measure is defined as the Euclidean norm of the $y$ and $z$ error at end time $T_{end} = 1.0$.}\label{fig:imexkaps}
	\end{center}
\end{figure}

\subsection{Stability results} 

In this section, we examine stability of our newly developed scheme.  We pay particular attention to the expected behavior of the solver for convection-diffusion equations, in which case the convection terms would be treated explicitely and the diffusive terms would be treated implicitely, as is common in the literature.  As pointed out in \cite{Ascher1995,Ascher1997}, a suitable prototype equation for the convection diffusion equation
is
\begin{align*}
 w' = (\lambda + i \mu) w \equiv \lambda w + i \mu w, 
\end{align*}
with $\lambda \leq 0$ and $\mu > 0$. The first part is treated implicitly, as it corresponds to the discretization of a diffusive operator, while the second part is treated explicitly. For a more detailed explanation of the relation between this equation and the convection-diffusion equation, we refer to \cite{Ascher1995}. 
To cast this into the framework of Alg. \ref{alg:mdode}, we define 
\begin{align*} 
 \Phii(w) := \lambda w, \qquad \Phie(w) := i \mu w, 
\end{align*}
which then yields 
\begin{align*}
 \dot \Phii(w) = \lambda ( \lambda + i \mu )w, \qquad \dot \Phie(w) = i \mu (\lambda  + i \mu ) w. 
\end{align*}
\newcommand{\llambda}{\tilde \lambda}
\newcommand{\mmu}    {\tilde \mu}
We follow the steps done in \cite{Ascher1997} and define 
\begin{align*}
 \llambda :=  \lambda \dt \leq 0, \qquad \mmu := \mu \dt > 0. 
\end{align*}
The predictor step for Alg. \ref{alg:mdode} can hence be written as 
\begin{align*} 
 \alg w 0 = \frac{1+i\mmu + \frac{i \mmu \llambda} 2 - \frac{\mmu^2}2}{1 -\llambda + \frac{\llambda^2}2 + i \frac{\llambda \mmu} 2} w^n =: \Psi(\llambda,\mmu) w^n.
\end{align*}
As already noticed in \cite{Ascher1997} for IMEX Euler, for $z := \lambda + i \mu$ being on the imaginary axis, i.e., $\lambda = 0$, this can, for $\mu \neq 0$, never yield an unconditionally stable algorithm, as 
\begin{align*} 
 \left| {1+i\mmu - \frac{\mmu^2}2} \right|^2 = \left(1-\frac{\mmu^2}2\right)^2 + \mmu^2 = 1 + \frac{\mmu^4}4 > 1. 
\end{align*}
This is of course not surprising, as the algorithm reduces to a purely explicit time marching scheme. In the spirit of \cite{Ascher1997}, we keep the ratio of $\lambda$ and $\mu$ constant, i.e., we define 
\begin{align*} 
 \gamma := \frac{\lambda}{\mu} \leq 0,
\end{align*}
and investigate whether, for a given ratio of the implicit to explicit eigenvalues, $\gamma$, the algorithm is stable.  This may produces a restrictions on $\mmu \equiv \mu \dt$, which we can modify by changing the time step size. Any restrictions on $\mmu \equiv \mu \dt$ will in practice result in a timestep restriction.
\begin{lemma}
 If $\gamma \leq -1$ then the predictor $\alg w 0$ for the method is stable.  That is, for a single time step, we have
 \begin{align*}
  \|\alg w 0\| \leq \|w^n\|. 
 \end{align*}
\end{lemma}
\begin{proof}
 Note that $\gamma = \frac \lambda \mu = \frac \llambda \mmu$ and consider the expression
 \begin{align*} 
  \left|\Psi(\gamma \mmu, \mmu)\right| \leq 1,
 \end{align*}
 which is equivalent to
 \begin{equation*}
  \left(1-\frac{\mmu^2} 2 \right)^2 + \left(\mmu + \frac{\gamma \mmu^2} 2\right)^2 \leq \left(1-\gamma \mmu+ \frac{(\gamma \mmu)^2} 2\right)^2 + \left(\frac{\gamma \mmu^2} 2\right)^2.
\end{equation*}
This again reduces to
\begin{equation*}
 \frac{1-\gamma^4}{4} \mmu^4 + (\gamma^3 + \gamma)\mmu^3  -2\gamma^2\mmu^2 + 2 \gamma \mmu \leq 0. 
 \end{equation*}
 As $\mmu$ is positive and $\gamma \leq -1$, this proves the claim. 
\end{proof}

Similar analysis can of course be made for the full algorithm. Due to the technical difficulties that come with high-order polynomials in $\lambda$ and $\mu$, we restrict ourselves to a numerical investigation. In Fig. \ref{fig:stability}, the maximum allowable $\mmu$ is shown as a function of $\gamma$ both for the predictor and the full algorithm, where we have restricted ourselves to $k_{\max} = 2$. It can be seen that for $\gamma \rightarrow 0$, the maximum allowable timestep for the predictor tends to zero, while for the full algorithm, it tends to a fixed constant. This is due to the fact that the squared absolute value of the iteration function for $k_{\max}=2$ and $\gamma = 0$ (hence $\lambda = 0$) is given by 
\begin{align*}
 \frac{\mmu^6\,\left(\mmu^6+76\,\mmu^4+1392\,\mmu^2-7488\right)}{82944} + 1, 
\end{align*}    
which is smaller than one between $\mmu = 0$ and $\mmu = 2.075$. \textit{The full algorithm hence gives a significant improvement in stability compared to just the predictor.} 

Finally, we perform our analysis also for the `limiting' method, i.e., the method defined by
\begin{align} 
 \label{eq:limitalg}
 w^{n+1} = w^n + \frac{\dt} 2 \left(\Phi^n + \Phi^{n+1}\right) + \frac{\dt^2}{12} \left( \dot\Phi^{n} - \dot\Phi^{n+1}\right). 
\end{align}
In the case that $\alg w k$ converges with $k \rightarrow \infty$, the limit exactly coincides with $w^{n+1}$ defined in \eqref{eq:limitalg}. In terms of $\llambda$ and $\mmu$, the iteration is given by 
\begin{align*} 
 w^{n+1} = \frac{1 + \frac 1 2 (\llambda + i\mmu) + \frac 1 {12} (\llambda + i\mmu)^2}{1 - \frac 1 2 (\llambda + i\mmu) + \frac 1 {12} (\llambda + i\mmu)^2}w^n =: \Theta(\llambda, \mmu) w^n. 
\end{align*}
For this fully implicit method, we have the following result:
\begin{lemma}
 If $\gamma < 0$, the the method defined in \eqref{eq:limitalg} has an amplification factor that satisfies
 \begin{align*} 
  |\Theta(\gamma \mmu, \mmu)| \leq 1.
 \end{align*}
\end{lemma}
\begin{proof}
 The claim is equivalent to 
 \begin{align*} 
  \left(1 + \frac 1 2 \gamma \mmu + \frac 1 {12} (\gamma^2 \mmu^2 - \mmu^2)\right)^2 + \left(\frac 1 2 \mmu + \frac 1 6 \gamma \mmu^2\right)^2 \leq 
  \left(1 - \frac 1 2 \gamma \mmu + \frac 1 {12} (\gamma^2 \mmu^2 - \mmu^2)\right)^2 + \left(-\frac 1 2 \mmu + \frac 1 6 \gamma \mmu^2\right)^2,
 \end{align*}
 which is then again equivalent to 
 \begin{align*} 
  \gamma \mmu \left( 12 + \gamma^2 \mmu^2 + \mmu^2\right) \leq 0.
 \end{align*}
This is true for any $\gamma < 0$. 

\end{proof}

\begin{figure}[h]
	\begin{center}
		\begin{tikzpicture}[scale=0.9]
		\begin{loglogaxis}[xlabel={$-\gamma = -\frac{\lambda}{\mu}$},ylabel={Maximum stable $\mmu$},grid=major,legend style={at={(0.12,0.98)},anchor=north west,font=\footnotesize,rounded corners=2pt}]
		\addplot table[x index = 0, y index =  1] {results_pred.dat};
		\addplot table[x index = 0, y index =  1] {results_k2.dat};
		\legend{Predictor,{Full algorithm ($k_{\max}=2$)}}
		\end{loglogaxis}
		\end{tikzpicture}		
		\caption{Maximum allowable $\mmu$ to guarantee stability depending on $\gamma$. Note that a restriction on $\mmu \equiv \mu\dt$ will in practice yield a restriction on the timestep $\dt$.}\label{fig:stability}
	\end{center}
\end{figure}
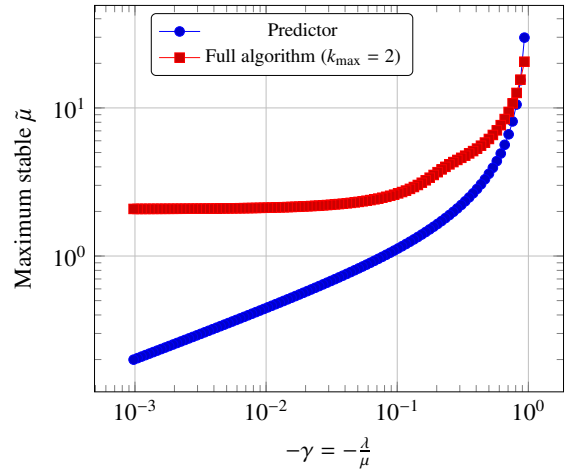

The finding here is that the limiting solver defined in \eqref{eq:limitalg} is numerically stable for any choice of time step.

\section{Conclusion and Outlook}
\label{sec:conout}
In this work we have presented a novel time integrator, featuring aspects of the spectrally deferred correction, implicit-explicit (IMEX) time integrators, and the multiderivative class of ODE integrators. We have shown that the method is asymptotically preserving and we have presented results for a class of test cases from the literature.  In addition, stability results for this solver have been investigated based on a prototype equation for convection-diffusion PDEs. 

There are many extensions of this work that can be found. Next steps include investigating the application of this solver to partial differential equations of the singularly perturbed type.  In particular, we are interested in the compressible Navier-Stokes equations at low Mach number \cite{KlMa81}. Suitable splittings have already been developed in literature, see, e.g., \cite{BispenIMEXSWE,Bispen2017222,CoDeKu12,DeTa,HaJiLi12,Kl95,ArNoLuMu12,RSIMEXFullEuler}. 
%
Similar to spectral deferred correction (SDC) methods, the proposed algorithm can certainly be parallelized in time, which, again in particular for PDEs, could make for a tremendous benefit in a parallel computing environment. Furthermore, extensions of this method that include variable orders of accuracy introduce the potential to investigate adaptive time stepping which would make for interesting results on their own merit.

\section*{Acknowledgements}
This study is the outcome of a research stay of D.C.\ Seal at the University of Hasselt, which was supported by the Special Research Fund (BOF) of Hasselt University. Additional funding came from the Office of Naval Research, grant number N0001419WX01523.


\section*{References}
\bibliographystyle{elsarticle-num}
\bibliography{ListPaper}


\end{document}